\documentclass[12pt]{amsart}
\usepackage{amsfonts, amssymb, amsmath, amsthm, eucal, latexsym, array, pictex}
\usepackage{fullpage}
\usepackage{verbatim}
\usepackage{color}
\definecolor{darkblue}{rgb}{0,0,.5}
\usepackage[colorlinks=true,linkcolor=darkblue,urlcolor=violet,citecolor=magenta]{hyperref}

\numberwithin{equation}{section}

\input{cyracc.def}

\newtheorem{thm}{Theorem}[section]
\newtheorem{lm}[thm]{Lemma} 

\theoremstyle{remark}
\newtheorem{ex}[thm]{Example}
\newtheorem{rmk}[thm]{Remark}
\newtheorem{conj}[thm]{Conjecture}

\theoremstyle{definition}

\newcommand{\gt}{\mathfrak}

\newcommand{\SL}{{\rm SL}}
\newcommand{\GL}{{\rm GL}}

\newcommand{\Spin}{{\rm Spin}}

\newcommand{\ind}{{\rm ind\,}}

\newcommand{\rk}{\mathrm{rk\,}}

\newcommand{\Lie}{\mathrm{Lie\,}}

\newcommand{\Ann}{\mathrm{Ann}}
\newcommand{\ad}{\mathrm{ad}}

\newcommand {\cj}{{\mathcal J}}

\newcommand{\esi}{\varepsilon}

\newcommand{\g}{\tilde{\mathfrak g}}

\newcommand {\cS}{{\mathcal S}}

\newcommand{\C}{{\mathbb C}}

\newcommand {\Z}{{\mathbb Z}}

\renewcommand{\le}{\leqslant}
\renewcommand{\ge}{\geqslant}

\newfam\eusfam%
\font\euszw=eusm10 scaled 1200%
\font\eusac=eusm7 scaled 1200%
\font\eusacc=eusm7 scaled 1000%
\textfont\eusfam=\euszw\scriptfont\eusfam=\eusac%
\scriptscriptfont\eusfam=\eusacc%
\begin{document}
\hfill {\scriptsize  December 22, 2017} 
\vskip1ex

\title[Semi-direct products with free algebras of invariants]
{Some semi-direct products \\ with free algebras of symmetric invariants}
\author{Oksana Yakimova}
\address{Institut f\"ur Mathematik, Friedrich-Schiller-Universit\"at Jena, Jena, 07737, Deutschland}
\email{oksana.yakimova@uni-jena.de}
\thanks{The author is partially supported by the DFG priority programme SPP 1388 ``Darstellungstheorie" and 
by the Graduiertenkolleg GRK 1523 ``Quanten- und Gravitationsfelder".}
\maketitle
\maketitle


\begin{abstract}
Let $\gt g$ be a  complex reductive  Lie algebra
and $V$ the underling vector space of a finite-dimensional representation of $\gt g$.
Then  one can consider a new Lie algebra $\gt q=\gt g{\ltimes} V$, which is a semi-direct product 
of $\gt g$ and an Abelian ideal $V$. We outline several results on the algebra 
$\C[\gt q^*]^{\gt q}$ of symmetric invariants of $\gt q$ and describe all semi-direct 
products related to the defining representation of $\gt{sl}_n$ with
$\C[\gt q^*]^{\gt q}$ being a free algebra.  
\end{abstract}

\section{Introduction}

Let $Q$ be a connected complex algebraic group.
Set $\gt q=\Lie Q$. Then $\cS(\gt q)=\C[\gt q^*]$
and $\cS(\gt q)^{\gt q}=\C[\gt q^*]^{\gt q}=\C[\gt q^*]^Q$.  
We will call the latter object the {\it algebra of symmetric invariants} of $\gt q$.
An important property of $\cS(\gt q)^{\gt q}$ is that it is isomorphic to $Z{\bf U}(\gt q)$
as an algebra by a classical result of M.\,Duflo (here $Z{\bf U}(\gt q)$ is the centre of 
the universal enveloping algebra of $\gt q$).  

Let $\gt g$ be a reductive Lie algebra. Then
by the Chevalley restriction theorem 
$\cS(\gt g)^{\gt g}=\C[H_1,\ldots,H_{\rk\gt g}]$ is a  polynomial ring 
(in $\rk\gt g$ variables). 
A quest for non-reductive Lie algebras with a similar property has recently become a 
trend in invariant theory. Here we consider finite-dimensional representations 
$\rho: \gt g \to \gt{gl}(V)$  of $\gt g$ and the corresponding semi-direct products 
$\gt q=\gt g{\ltimes} V$. 
The Lie bracket on $\gt q$ is defined by
\begin{equation}\label{Lie-s}
[\xi+v,\eta+u]=[\xi,\eta]+\rho(\xi)u-\rho(\eta)v 
\end{equation}
for all $\xi,\eta\in\gt g$, $v,u\in V$. 
Let $G$ be a connected simply connected Lie group with $\Lie G=\gt g$. Then 
$\gt q=\Lie Q$ with $Q=G{\ltimes}\exp(V)$. 
 
It is easy to see that $\C[V^*]^G\subset \C[\gt q^*]^{\gt q}$ and therefore 
$\C[V^*]^G$ must be a polynomial ring if $\C[\gt q^*]^{\gt q}$ is, see \cite[Section~3]{z2}. 
Classification of the  representations of complex simple algebraic groups with 
free algebras of invariants was carried out
by G.\,Schwarz~\cite{gerry}
and  independently by  O.M.\,Adamovich and E.O.\,Golovina~\cite{devochki}. 
One such representation is the spin-representation of $\Spin_7$, which leads to
$Q=\Spin_7{\ltimes}\C^8$.
Here $\C[\gt q^*]^{\gt q}$ is a polynomial ring in $3$ variables generated by invariants 
of bi-degrees $(0,2)$, $(2,2)$, $(6,4)$ with respect to the decomposition 
$\gt q=\gt{so}_7{\oplus} \C^8$, see \cite[Proposition~3.10]{z2}. 

In this paper, we treat another example, $G=\SL_n$, $V=m(\C^n)^*{\oplus}k\C^n$ 
with $n\ge 2$, $m\ge 1$, $m\ge k$. Here $\C[\gt q^*]^{\gt q}$ is a polynomial ring 
in exactly the following three cases:
\begin{itemize}
\item $k=0$, $m\le n{+}1$, and $n \equiv t \pmod{m}$ with $t\in\{-1,0,1\}$; 
\item $m=k$, $k\in\{n{-}2,n{-}1\}$;
\item $n\ge m>k>0$ and $m{-}k$ divides $n{-}m$.   
\end{itemize}

We also briefly discuss semi-direct products arising as $\mathbb Z_2$-contractions of reductive Lie algebras. 

\section{Symmetric invariants  and generic stabilisers}\label{sec-2}

Let $\gt q=\Lie Q$ be an algebraic Lie algebra, $Q$ a connected algebraic group. 
The index  of $\gt q$ is defined as 
$$\ind\gt q=\min_{\gamma\in\gt q^*} \dim\gt q_\gamma, $$ 
where $\gt q_\gamma$ is the stabiliser of $\gamma$ in $\gt q$. 
In view of Rosenlicht's theorem, $\ind\gt q={\rm tr.deg}\,\C(\gt q^*)^Q$.
In case $\ind\gt q=0$,  we have $\C[\gt q^*]^{\gt q}=\C$.
For a reductive $\gt g$, $\ind\gt g=\rk\gt g$. Recall that $(\dim\gt g+\rk\gt g)/2$ is the dimension 
of a Borel subalgebra of $\gt g$. For $\gt q$, set ${\bf b}(\gt q):=(\ind\gt q+\dim\gt q)/2$. 

Let $\{\xi_i\}$ be a basis of $\gt q$ and ${\mathcal M}(\gt q)=([\xi_i,\xi_j])$ the structural matrix with entries in $\gt q$. 
This is a  skew-symmetric matrix of rank $\dim\gt q-\ind\gt q$. Let us take Pfaffians of the principal 
minors of  ${\mathcal M}(\gt q)$ of size $\rk {\mathcal M}(\gt q)$ and let ${\bf p}={\bf p}_{\gt q}$
be their greatest common divisor. Then   ${\bf p}$ is called the
 {\it fundamental semi-invariant} of $\gt q$.
The zero set of ${\bf p}$ is the maximal divisor in the so called {\it singular set}
$$
\gt q^*_{\rm sing}=\{\gamma\in\gt q^* \mid \dim\gt q_\gamma>\ind \gt q\},
$$
of $\gt q$. 
Since $\gt q^*_{\rm sing}$ is clearly a $Q$-stable subset, ${\bf p}$ is indeed a semi-invariant,
$Q{\cdot}{\bf p}\subset \C{\bf p}$. 
One says that $\gt q$ has 
the ``codim-2" property (satisfies the ``codim-2" condition), if 
$\dim\gt q ^*_{\rm sing}\le \dim\gt q-2$ or equivalently if ${\bf p}=1$.

Suppose that $F_1,\ldots,F_r\in\cS(\gt q)$ are homogenous algebraically independent 
polynomials. The {\it Jacobian locus} ${\cj}(F_1,\ldots,F_r)$ of these polynomials consists of all
$\gamma\in \gt q^*$  such that the differentials 
$d_\gamma F_1,\ldots,d_\gamma F_r$ are linearly dependent. 
In other words, $\gamma\in {\cj}(F_1,\ldots,F_r)$ if and only if 
$(dF_1\wedge\ldots\wedge dF_r)_\gamma=0$.
The set ${\mathcal J}(F_1,\ldots,F_r)$ is a proper Zariski closed subset of $\gt q^*$.
Suppose that ${\cj}(F_1,\ldots,F_r)$ does not contain divisors.
Then by the characteristic zero version of 
a result of Skryabin, see \cite[Theorem~1.1]{ppy}, $\C[F_1,\ldots,F_r]$ is an 
algebraically closed subalgebra of $\cS(\gt q)$, each  $H\in \cS(\gt q)$ that is 
algebraic over $\C(F_1,\ldots,F_r)$ is contained in
$\C[F_1,\ldots,F_r]$.

\begin{thm}[cf. {\cite[Section~5.8]{JSh}}]\label{thm-sum}
Suppose that 
${\bf p}_{\gt q}=1$ and suppose that 
$H_1,\ldots,H_r\in\cS(\gt q)^{\gt q}$ are homogeneous algebraically independent 
polynomials such that $r=\ind\gt q$ and $\sum_{i=1}^r \deg H_i={\bf b}(\gt q)$.
Then  $\cS(\gt q)^{\gt q}=\C[H_1,\ldots,H_r]$ is a polynomial ring in $r$ generators. 
\end{thm}
\begin{proof}
Under our assumptions $\cj(H_1,\ldots,H_r)=\gt q^*_{\rm sing}$, see \cite[Theorem~1.2]{ppy} and \cite[Section~2]{jems}. Therefore $\C[H_1,\ldots,H_r]$ is an algebraically closed subalgebra of
$\cS(\gt q)$ by \cite[Theorem~1.1]{ppy}. Since ${\rm tr.deg}\,\cS(\gt q)^{\gt q}\le r$, each symmetric $\gt q$-invariant is algebraic over $\C[H_1,\ldots,H_r]$ and hence is contained in it. 
\end{proof}

For semi-direct products, we have some specific approaches to the  symmetric invariants.  
Suppose now that $\gt g=\Lie G$ is a reductive Lie algebra, $G$ is connected, and 
$\gt q=\gt g{\ltimes} V$, where $V$ is a finite-dimensional $G$-module. 

The vector space decomposition $\gt q=\gt g{\oplus} V$ leads to 
$\gt q^*=\gt g{\oplus} V^*$, where we identify $\gt g$ with $\gt g^*$. 
Each element $x\in V^*$ is considered as a point of $\gt q^*$ that is zero on $\gt g$.
We have $\exp(V){\cdot}x=\ad^*(V){\cdot}x+x$, where each element of $\ad^*(V){\cdot}x$ is zero on $V$. 
Note that $\ad^*(V){\cdot}x\subset \Ann(\gt g_x)\subset \gt g$ and $\dim(\ad^*(V){\cdot}x)$ is equal to
$\dim(\ad^*(\gt g){\cdot}x)=\dim\gt g-\dim\gt g_x$. Therefore $\ad^*(V){\cdot}x= \Ann(\gt g_x)$. 
 
The decomposition $\gt q=\gt g{\oplus} V$ defines also a bi-grading on $\cS(\gt q)$ and clearly 
$\cS(\gt q)^{\gt q}$ is a bi-homogeneous subalgebra, cf. \cite[Lemma~2.12]{z2}. 
 
A statement is true for a ``generic x" if and only if this statement is true for all 
points of a  non-empty open subset.

\begin{lm}\label{V-inv}
A function $F\in\C[\gt q^*]$ is a $V$-invariant if and only if $F(\xi+\ad^*(V){\cdot}x,x)=F(\xi,x)$ for 
generic $x\in V^*$ and any $\xi\in\gt g$. 
\end{lm}
\begin{proof}
Condition of the lemma guaranties that for each $v\in V$, $\exp(v){\cdot}F=F$ on a non-empty open 
subset of $\gt q^*$. Hence $F$ is a $V$-invariant.   
\end{proof}

For $x\in V^*$, let $\varphi_x\! : \C[\gt q^*]^Q\to \C[\gt g{+}x]^{G_x{\ltimes}\exp(V)}$ be the restriction map. By \cite[Lemma~2.5]{z2} $\C[\gt g{+}x]^{G_x{\ltimes}\exp(V)}\cong \cS(\gt g_x)^{G_x}$. Moreover, if we identify $\gt g{+}x$ with $\gt g$ choosing $x$ as the origin, then 
$\varphi_x(F)\in\cS(\gt g_x)$ for any $\gt q$-invariant $F$ \cite[Section~2]{z2}. 
Under certain assumptions on $G$ and $V$ the restriction map $\varphi_x$ is surjective, more details 
will be given shortly.

There is a non-empty open subset $U\subset V^*$ such
that the stabilisers $G_x$ and $G_y$ are conjugate in $G$ for any pair of points $x,\,y\in U$
see e.g. \cite[Theorem~7.2]{VP}. 
Any representative of the conjugacy class $\{hG_xh^{-1}\mid h\in G, x\in U\}$
is said to be a {\it a generic stabliser} of the $G$-action on $V^*$.

There is one easy to handle case, $\gt g_x=0$ for a generic $x\in V^*$.
Here $\C[\gt q^*]^{Q}=\C[V^*]^{G}$, see e.g. \cite[Example~3.1]{z2}, and 
$\xi+y\in\gt q^*_{\rm sing}$ only if $\gt g_y\ne 0$, where $\xi\in\gt g$, $y\in V^*$. 
The case  $\ind\gt g_x=1$ is more involved.

\begin{lm}\label{index-1-sur}
Assume that $G$ has no proper semi-invariants in $\C[V^*]$. 
Suppose that $\ind\gt g_x=1$, $\cS(\gt g_x)^{\gt g_x}\ne \C$,  and the map $\varphi_x$ is surjective for generic $x\in V^*$.
Then $\C[\gt q^*]^{\gt q}=\C[V^*]^G[F]$, where $F$ is a bi-homogeneous preimage of 
a generator of $\cS(\gt g_x)^{G_x}$ that is not divisible by any non-constant $G$-invariant in $\C[V^*]$.    
\end{lm}
\begin{proof}
If we have a Lie algebra of index $1$, in our case $\gt g_x$, then the algebra of its symmetric invariants is a polynomial ring. There are many possible explanations of this fact. One of them is the following. Suppose that two 
non-zero 
homogeneous polynomials $f_1,f_2$ are algebraically dependent. Then $f_1^a=c f_2^b$ for some coprime integers $a,b>0$
and some $c\in\mathbb C^{^\times}$. If $f_1$ is an invariant, then so is 
a polynomial function $\sqrt[b]{f_1}=\sqrt[ab]{c}\sqrt[a]{f_2}$. 

Since $\cS(\gt g_x)^{\gt g_x}\ne \C$, it is generated by some homogeneous $f$. The group $G_x$ has finitely many connected components, hence $\cS(\gt g_x)^{G_x}$ is generated by a suitable power of $f$, say ${\bf f}=f^d$. 

Let $F\in\C[\gt q^*]^Q$ be a preimage of ${\bf f}$. Each its bi-homogeneous component is again a $\gt q$-invariant. Without loss of generality we may assume that $F$ is bi-homogenous. Also if $F$ is divisible by 
some non-scalar  $H\in\C[V^*]^G$, then we replace $F$ with $F/H$ and repeat the process as long as possible.

Whenever $G_y$ (with $y\in V^*$) is conjugate to $G_x$ and $\varphi_y(F)\ne 0$, 
$\varphi_y(F)$ is a $G_y$-invariant of the same degree as ${\bf f}$ and therefore is a generator 
of $\cS(\gt g_y)^{G_y}$. 
Clearly $\C(V^*)^G[F]\subset \C[\gt q^*]^Q{\otimes}_{\C[V^*]^G}\C(V^*)^G=:{\mathcal A}$ and 
${\mathcal A}\subset \cS(\gt g){\otimes}\C(V^*)^G$. If ${\mathcal A}$ contains a homogeneous in $\gt g$ polynomial $T$ that is not proportional (over $\C(V^*)^G$) to a power of $F$, then $\varphi_u(T)$ is not proportional  to a power of  $\varphi_u(F)$ for generic $u\in V^*$.
But $\varphi_u(T)\in\cS(\gt g_u)^{G_u}$. This implies that 
${\mathcal A}=\C(V^*)^G[F]$. 
It remains to notice that $\C(V^*)^G={\rm Quot}\,\C[V^*]^G$, since $G$ has no proper semi-invariants 
in $\C[V^*]$, and by the same reason $\C(V^*)^G[F]\cap \C[\gt q]=\C[V^*]^G[F]$ in case 
$F$ is  not divisible by any non-constant $G$-invariant in $\C[V^*]$.
\end{proof}

It is time to recall the Ra{\"i}s' formula \cite{r} for the index of a semi-direct product: 
\begin{equation}\label{ind-sum}
\ind\gt q=\dim V-(\dim\gt g-\dim\gt g_x)+\ind\gt g_x \ \text{ with $x\in V^*$ generic. } 
\end{equation}

\begin{lm}\label{V-deg}
Suppose that $H_1,\ldots,H_r\in\cS(\gt q)^Q$ are 
homogenous polynomials such that $\varphi_x(H_i)$ with $i\le \ind\gt g_x$ freely generate 
$\cS(\gt g_x)^{G_x}=\cS(\gt g_x)^{\gt g_x}$ for generic $x\in V^*$ and $H_j\in\C[V^*]^G$ for $j>\ind\gt g_x$; 
and suppose that $\sum\limits_{i=1} ^{\ind\gt g_x} \deg_{\gt g} H_i={\bf b}(\gt g_x)$. Then 
$\sum\limits_{i=1}^r \deg H_i={\bf b}(\gt q)$ if and only if
$\sum\limits_{i=1}^r \deg_V H_i=\dim V$. 
\end{lm}
\begin{proof}
In view of the assumptions, we have 
$\sum\limits_{i=1}^r \deg H_i={\bf b}(\gt g_x)+\sum\limits_{i=1}^r \deg_V H_i$.  Further, by Eq.~\eqref{ind-sum}
$$
\begin{array}{l}
 {\bf b}(\gt q)=(\dim\gt q+\dim V-(\dim\gt g-\dim\gt g_x)+\ind\gt g_x)/2= \\
\quad  =\dim V+(\dim\gt g_x+\ind\gt g_x)/2={\bf b}(\gt g_x) + \dim V.
\end{array}
$$
The result follows. 
\end{proof}

From now on suppose that $G$ is semisimple. Then both  $G$ and  $Q$ have only trivial 
characters and hence cannot have proper semi-invariants. 
In particular, the fundamental semi-invariant is an invariant. We also 
have ${\rm tr.deg}\,\cS(\gt q)^{\gt q}=\ind\gt q$. Set $r=\ind\gt q$ and 
let $x\in V^*$ be generic. If $\C[\gt q^*]^Q$ is a polynomial ring, then there are bi-homogenous generators $H_1,\ldots ,H_r$  such that 
$H_i$ with $i>\ind\gt g_x$ freely generate $\C[V^*]^G$ and the invariants 
$H_i$ with $i\le \ind\gt g_x$ are {\it mixed}, they have positive degrees in $\gt g$ and $V$.  

\begin{thm}[{\cite[Theorem~5.7]{JSh}\&\cite[Proposition~3.11]{z2}}]\label{svob}
Suppose that $G$ is semisimple and $\C[\gt q^*]^{\gt q}$ is a polynomial ring
with homogeneous generators $H_1,\ldots,H_r$. Then 
\begin{itemize}
\item[({\sf i})] $\sum_{i=1}^r \deg H_i={\bf b}(\gt q)+\deg{\bf p}_{\gt q}$;
\item[({\sf ii})]   for   generic $x\in V^*$, the restriction map $\varphi_x\!: 
  \C[\gt q^*]^Q \to \C[\gt g{+}x]^{G_x{\ltimes}V}\cong\cS(\gt g_x)^{G_x}$ 
is surjective, $\cS(\gt g_x)^{G_x}=\cS(\gt g_x)^{\gt g_x}$, and $\cS(\gt g_x)^{G_x}$ is a polynomial ring in $\ind\gt g_x$ variables. 
\end{itemize}
\end{thm}

It is worth mentioning that $\varphi_x$ is also surjective for stable actions.  
An action of $G$ on $V$ 
is called {\it stable} if generic $G$-orbits in $V$ are closed, for more details see \cite[Sections~2.4\& 7.5]{VP}. 
By \cite[Theorem~2.8]{z2} $\varphi_x$ is surjective for generic $x\in V^*$ if the $G$-action on $V^*$ is stable.

\section{$\Z/2\Z$-contractions} 

The initial motivation for studying symmetric invariants of semi-direct products  
was related to  a conjecture of D.\,Panyushev  on $\Z_2$-contractions of reductive Lie algebras. 
The results of \cite{z2}, briefly outlined in Section~\ref{sec-2}, has settled the problem. 

Let $\gt g=\gt g_0{\oplus}\gt g_1$ 
be a {\it symmetric decomposition}, i.e., a $\Z/2\Z$-grading of $\gt g$.  
A semi-direct product, $\g=\gt g_0{\ltimes}\gt g_1$, where $\gt g _1$ is an 
Abelian ideal, can be seen as a {\it contraction}, in this case a {\it $\Z_2$-contraction},
of $\gt g$. For example, starting with a symmetric pair $(\gt{so}_{n+1},\gt{so}_n)$, 
one arrives at $\g=\gt{so}_n{\ltimes}\C^n$.  In \cite{Dima06}, it was conjectured
that $\cS(\g)^{\g}$ is a polynomial ring (in $\rk\gt g$ variables).

\begin{thm}[\cite{Dima06,jems}, \cite{z2}] Let $\tilde{\gt g}$ be a $\Z_2$-contraction of a reductive
Lie algebra $\gt g$. Then $\cS(\tilde{\gt g})^{\g}$ is a  polynomial ring 
(in $\rk\gt g$ variables)  if and only if the restriction homomorphism $\C[\gt g]^{\gt g}\to \C[\gt g_1]^{\gt g_0}$ is 
surjective.
\end{thm}

If we are in one of the ``surjective" cases, then one can describe the 
generators of $\cS(\g)^{\g}$.  
Let $H_1,\ldots,H_r$ be suitably chosen homogeneous generators of $\cS(\gt g)^{\gt g}$ and 
let $H_i^\bullet$ be the bi-homogeneous (w.r.t. $\gt g=\gt g_0{\oplus}\gt g_1$) 
component of $H_i$ of the highest $\gt g_1$-degree. Then 
$\cS(\g)^{\g}$ is freely generated by the polynomials $H_i^\bullet$ (of course, providing 
the restriction homomorphism $\C[\gt g]^{\gt g}\to \C[\gt g_1]^{\gt g_0}$ is 
surjective) \cite{Dima06,jems}. 

Unfortunately,  this construction of generators cannot work if the restriction homomorphism is not surjective,
see \cite[Remark~4.3]{Dima06}. 
As was found out by S.\,Helgason \cite{Helg}, there are four ``non-surjective" irreducible symmetric pairs, 
namely, $(E_6, F_4)$, $(E_7, E_6{\oplus}\C)$,  
$(E_8, E_7{\oplus}\gt{sl}_2)$, and $(E_6,\gt{so}_{10}{\oplus}\gt{so}_2)$.
The approach to semi-direct products developed in \cite{z2} showed that Panyushev's conjecture 
does not hold for them. 
Next we outline some ideas of the proof. 

Let $G_0\subset G$ be a connected subgroup with $\Lie G_0=\gt g_0$. 
Then $G_0$ is reductive, it acts on $\gt g_1\cong\gt g_1^*$, and this action is stable. 
Let $x\in\gt g_1$ be a generic element and $G_{0,x}$ be its stabiliser in $G_0$.
The groups $G_{0,x}$ are reductive and they are  known for all symmetric pairs. 
In particular, $\cS(\gt g_{0,x})^{G_{0,x}}$ is a polynomial ring. It is also known that 
$\C[\gt g_1]^{G_0}$ is a polynomial ring. By \cite{Dima06} $\g$ has the ``codim-2" property and 
$\ind\g=\rk\gt g$. 

Making use of the surjectivity of $\varphi_x$ one can show that if 
$\C[\g^*]^{\g}$ is freely generated by some $H_1,\ldots,H_r$, then necessary 
$\sum\limits_{i=1}^r \deg H_i>{\bf b}(\g)$ for $\g$ coming from one of the ``non-surjective" pairs \cite{z2}. 
In view of some results from \cite{JSh} this leads to a contradiction. 

Note that in case of $(\gt g,\gt g_0)=(E_6, F_4)$, $\gt g_0=F_4$ is simple and 
$\g$ is a semi-direct product of $F_4$ and $\C^{26}$,
which, of course, comes from one of the representations in Schwarz's list \cite{gerry}. 
  
\section{Examples related to the defining representation of $\gt{sl}_n$} 

Form now assume that 
$\gt g=\gt{sl}_n$ and  
$V=m (\C^n)^*{\oplus}k \C^n$ with $n\ge 2$, $m\ge 1$, $m\ge k$. 
According to \cite{gerry} $\C[V]^G$ is a polynomial ring if either $k=0$ and 
$m\le n{+}1$ or $m\le n$, $k\le n{-}1$. One finds also the description of the
generators of  $\C[V^*]^G$ and their degrees in \cite{gerry}. In this section, we classify all cases, 
where 
$\C[\gt q^*]^{\gt q}$ is a polynomial ring and for each of them give the fundamental semi-invariant. 
  
\begin{ex}\label{ex-0}  
Suppose that either   $m\ge n$ or $m=k=n{-}1$. Then $\gt g_x=0$ for generic $x\in V^*$ and therefore 
$\C[\gt q^*]^Q=\C[V^*]^G$, i.e., $\C[\gt q^*]^Q$ is a polynomial ring if and only if $
\C[V^*]^G$ is. The latter takes place for $(m,k)=(n{+}1,0)$,  for $m=n$ and any $k<n$,
as well as for $m=k=n{-}1$. Non-scalar fundamental semi-invariants appear here only for  
\begin{itemize}
\item $m=n$, where ${\bf p}$ is given by $\det (v)^{n{-}1{-}k}$ with $v\in n\C^n$;
\item $m=k=n{-}1$, where ${\bf p}$ is the sum of the principal $2k{\times}2k$-minors of 
$$
\left(
  \begin{array}{c|c} 
   0 & v \\
   \hline
   w & 0 \\
\end{array}\right) \text{ with } v\in k\C^n, w\in k(\C^n)^*. 
$$

\end{itemize}
\end{ex}

In the rest of the section, we assume that $\gt g_x\ne 0$ for generic $x\in V^*$. 

\subsection{The case $k=0$}
Here the ring of $G$-invariants on $V^*$ is generated 
by 
$$\{\Delta_I \mid I\subset \{1,\ldots,m\},  |I|=n\} \ \ \text{\cite[Section~9]{VP}}, $$ where 
each $\Delta_I(v)$ is the determinant of the corresponding sub-matrix of $v\in V^*$. 
The generators are algebraically independent if and only if $m\le n+1$, see also \cite{gerry}. 

We are interested only in  $m$ that are smaller than $n$. Let $n=qm+r$, where $0 < r \le  m$, and let $I\subset \{1,\ldots, m\}$ be 
a subset  of cardinality $r$. By choosing the corresponding $r$ columns of $v$ we get a matrix $w=v_I$. 
Set
\begin{equation}\label{eq-F-1}
F_I (A,v):=\det\left(v | Av |  \ldots  | A^{q{-}1}v  |A^q w\right), \ \text{ where } A\in\gt g, v\in V^*. 
\end{equation}
Clearly each  $F_I$ is an $\SL_n$-invariant. Below we will see that they are also $V$-invariants. 
If $r=m$, then 
there is just one invariant, $F=F_{\{1,\ldots,m\}}$. If $r$ is either $1$ or $m-1$, we get $m$ invariants.  

\begin{lm}\label{F-V} Each $F_I$ defined by Eq.~\eqref{eq-F-1} is a $V$-invariant. 
\end{lm}
\begin{proof}
According to Lemma~\ref{V-inv} we have to show that $F_I(\xi{+}\ad^*(V) {\cdot} x,x)=F(\xi,x)$ for generic $x\in V^*$ and any $\xi\in\gt{sl}_n$. 
Since $m < n$, there is an open $\SL_n$-orbit in $V^*$ and we can take $x$ as $E_m$. 
Let $\gt p\subset\gt{gl}_n$ be the standard  parabolic subalgebra  corresponding to the composition 
$(m,n{-}m)$ and let $\gt n_-$  be the nilpotent radical of the opposite parabolic. 
Each element (matrix) $\xi\in\gt{gl}_n$ is a sum $\xi=\xi_-+\xi_p$ with $\xi_-\in\gt n_-$, $\xi_p\in\gt p$. 
In this notation 
$F_I(A,E_m)=\det\left( A_-  |(A^2)_- \ldots  | (A^{q{-}1})_-| (A^q)_{-,I} \right)$. 

Let $\alpha=\alpha_A$ and $\beta=\beta_A$ be 
$m{\times}m$ and $(n{-}m){\times}(n{-}m)$-submatrices of $A$ standing in the upper left and lower right corner, respectively. 
Then $(A^{s{+}1})_-=\sum_{t=0}^{s}  \beta^t A_- \alpha^{s{-}t}$. 
Each column of $A_-\alpha$ is a linear combination of columns of $A_-$  and each column of
$\beta^{t} A_-\alpha^{j{+}1}$  
is a linear combination of columns of $\beta^{t} A_-\alpha^{j}$. Therefore 
\begin{equation}\label{r-1}
F_I(A,E_m)=\det\left( A_- |\ldots | (A^{q{-}1})_- | (A^q)_{-,I} \right)=\det\left( A_-|\beta A_-|\ldots |\beta^{q{-}2}A_-| \beta^{q{-}1}A_{-,I}\right).
\end{equation}

Notice that $\gt g_x\subset\gt p$ and the nilpotent radical of $\gt p$ is contained in $\gt g_x$ (with $x=E_m$). 
Since $\ad^*(V){\cdot}x=\Ann(\gt g_x)=\gt g_x^{\perp}\subset \gt g$ (after the identification $\gt g\cong\gt g^*$), $A_-=0$ for any $A\in\gt g_x^{\perp}$; and we have $\beta_A=c E_{n{-}m}$ with $c\in\C$ for this  $A$. An easy observation is that 
$$
\det\left(\xi_- | (\beta_\xi{+}cE_{n{-}m})\xi_- | \ldots | (\beta_\xi{+}cE_{n{-}m})^{q-1} \xi_{-,I} \right)=
  \det\left(\xi_- | \beta_\xi \xi_- | \ldots | \beta_\xi^{q-1} \xi_{-,I} \right).
$$
Hence
$F_I(\xi+A,E_m)=F_I(\xi,E_m)$ for all $A\in\ad^*(V){\cdot}E_m$ and all $\xi\in\gt{sl}_n$.  
\end{proof}

\begin{thm}\label{th-0}
Suppose that $\gt q=\gt{sl}_n{\ltimes} m(\C^n)^*$. Then $\C[\gt q^*]^Q$ is a polynomial ring 
if and only if $m\le n+1$ and $m$ divides either $n{-}1$, $n$ or $n{+}1$. Under these assumptions on 
$m$,  ${\bf p}_{\gt q}=1$ exactly then, when $m$ divides either $n{-}1$ or $n{+}1$. 
\end{thm}
\begin{proof}
Note  that the statement is true for $m\ge n$ by Example~\ref{ex-0}.
Assume that $m\le n{-}1$. 
Suppose that $n=mq+r$ as above.  
A generic stabiliser in $\gt g$ is  $\gt g_x=\gt{sl}_{n-m}{\ltimes} m\C^{n-m}$. On the group level it is connected. 
Notice that $\ind\gt g_x={\rm tr.deg}\,\cS(\gt g_x)^{G_x}$, since $G_x$ has no non-trivial characters. 
Note also that $\C[V^*]^G=\C$, since $m<n$. 
If $\C[\gt q^*]^Q$ is a polynomial ring, then so is $\C[\gt g_x^*]^{G_x}$ by Theorem~\ref{svob}({\sf ii}) and either 
$n{-}m=1$ or, arguing by induction, $n{-}m \equiv t \pmod{m}$ with $t\in\{-1,0,1\}$. 

Next we show that 
the ring of symmetric invariants is freely generated by the polynomials $F_I$ for the indicated $m$. 
Each element $\gamma\in\gt g_x^*$ can be presented as $\gamma=\beta_0{+} A_-$, where 
$\beta_0\in\gt{sl}_{n{-}m}$. Each restriction $\varphi_x(F_I)$ can be regarded as an element of $\cS(\gt g_x)$. Eq.~\eqref{r-1} combined with Lemma~\ref{F-V} and the observation that 
$\gt g_x^*\cong \gt g/\Ann(\gt g_x)$ shows that 
$\varphi_x(F_I)$ is either  $\Delta_I$ of $\gt g_x$ (in case $q=1$, where $F_I(A,E_m)=\det A_{-,I}$) or 
$F_I$ of $\gt g_x$. Arguing by induction on $n$, we prove that the restrictions
$\varphi_x(F_I)$ freely generate $\cS(\gt g_x)^{\gt g_x}$ for $x=E_m$ (i.e., for a generic point in $V^*$). 
Notice that $n{-}m=(q{-}1)m+r$. 

The group $\SL_n$ acts on $V^*$ with an open orbit $\SL_n{\cdot}E_m$. Therefore 
the restriction map $\varphi_x$ is injective. By the inductive hypothesis it is also 
surjective and therefore is an isomorphism. This proves that 
the polynomials $F_I$ freely generate $\C[\gt q^*]^Q$. 

If $m$ divides $n$, then 
$\C[\gt q^*]^Q=\C[F]$ and  the fundamental semi-invariant is a power of $F$. 
As follows from the equality in Theorem~\ref{svob}({\sf i}), 
${\bf p}=F^{m-1}$.

Suppose that $m$ divides either $n-1$ or $n+1$. Then  we have $m$ different invariants 
$F_I$.  
By induction on $n$, $\gt g_x$ has the ``codim-2" property, therefore the sum of 
$\deg\varphi_x(F_I)$ is equal 
to ${\bf b}(\gt g_x)$ by Theorem~\ref{svob}({\sf i}). 
The sum of $V$-degrees is $m{\times}n=\dim V$ and hence by Lemma~\ref{V-deg}  
$\sum \deg F_I={\bf b}(\gt q)$. Thus, $\gt q$ has the ``codim-2" property. 
\end{proof}

\begin{rmk}
Using induction on $n$  one can show that  the restriction map $\varphi_x$ is an isomorphism 
for all $m<n$. Therefore the polynomials $F_I$ generate $\C[\gt q^*]^Q$ for all $m<n$.
\end{rmk}

\subsection{The case $m=k$} Here $\C[V^*]^G$ is a polynomial ring if and only if $k\le n{-}1$, a generic stabiliser is $\gt{sl}_{n-k}$, and  the $G$-action on $V\cong V^*$ is stable. We assume that 
$\gt g_x\ne 0$ for generic $x\in V^*$ and therefore $k\le n{-}2$.  

For an $N{\times}N$-matrix $C$, let $\Delta_i(C)$ with $1\le i\le N$ be coefficients of its 
characteristic polynomial, each $\Delta_i$ being a homogeneous polynomial of degree $i$.  
Let $\gamma=A{+}v{+}w\in\gt q^*$ with $A\in\gt g$, $v\in k\C^n$, $w\in k(\C^n)^*$. Having these objects we form an 
$(n+k){\times}(n+k)$-matrix 
$$
Y_\gamma:=\left(
  \begin{array}{c|c} 
   A & v \\
   \hline
   w & 0 \\
\end{array}\right)
$$
and set $F_i(\gamma)=\Delta_i(Y_\gamma)$ for each $i\in \{2k{+}1,2k{+}2, 2k{+}3, \ldots, n{+}k\}$. 
Each $F_i$ is an $\SL_n{\times}\GL_k$-invariant. 
Unfortunately, these polynomials are not $V$-invariants.  

\begin{rmk}
If we repeat the same construction for $\tilde{\gt q}=\gt{gl}_n{\ltimes}V$ with $k\le n{-}1$, then 
$\C[\tilde{\gt q}^*]^{\tilde Q}=\C[V^*]^{\GL_n}[\{F_i\mid 2k{+}1\le i\le n{+}k\}]$ and it is a polynomial ring 
in $\ind\tilde{\gt q}=n{-}k{+}k^2$ generators. 
\end{rmk}

\begin{thm}\label{m=k}
Suppose that $m=k\le n{-}1$. Then $\C[\gt q^*]^{\gt q}$ 
is a polynomial ring if and only if $k\in\{n{-}2,n{-}1\}$. In case $k=n{-}2$, $\gt q$ has the ``codim-2"
property.
\end{thm}
\begin{proof}
Suppose that $k=n{-}2$. Then a generic stabiliser  $\gt g_x=\gt{sl}_2$ is of index $1$ and since the $G$-action on $V$ is stable, $\C[\gt q^*]^{\gt q}$ has to be a polynomial ring by \cite[Example~3.6]{z2}. 
Checking the ``codim-2" property is a routine calculation. Alternatively one can show that 
the unique mixed generator is of the form $F_{2k{+}2}H_{2k}-F_{2k{+}1}^2$, where $H_{2k}$ is a certain 
$\SL_n{\times}\GL_k$-invariant on $V$ of degree $2k$ and then see that the sum of degrees is ${\bf b}(\gt q)$.

Suppose that $0<k<n{-}2$ and assume that  
$\cS(\gt q)^{\gt q}$ is a polynomial ring. Then there are bi-homogeneous generators 
${\bf h}_2,\ldots,{\bf h}_{n-k}$ 
of $\C[\gt q^*]^Q$ over $\C[V^*]^G$ 
such that their restrictions to $\gt g{+}x$ form a generating set 
of $\cS(\gt g_x)^{\gt g_x}$ for a generic $x$ (with $\gt g_x\cong\gt{sl}_{n{-}k}$) Theorem~\ref{svob}({\sf ii}). 
In particular, $\deg_{\gt g} {\bf h}_t=t$. 

Take $\g=(\gt{sl}_n{\oplus}\gt{gl}_k)\ltimes V$, which is a  $\Z_2$-contraction of $\gt{sl}_{n{+}k}$. 
Then $\gt q$ is a Lie subalgebra of  $\g$. 
Note that $\GL_k$ acts on $\gt q$ via automorphisms and therefore we may assume that the $\C$-linear span of 
$\{{\bf h}_t\}$ is $\GL_k$-stable. By degree considerations, each ${\bf h}_t$ is an $\SL_k$-invariant as well. 
The Weyl involution of $\SL_n$ acts on $V$ and has to preserve 
each line $\C{\bf h}_t$. Since this involution interchanges $\C^n$ and $(\C^n)^*$,  
each  ${\bf h}_t$ is also a $\GL_k$-invariant. Thus,  $$\cS(\gt q)^{\gt q}=\cS(\gt q)^{\g}=\cS(\gt q)\cap \cS(\g)^{\g}.$$

Since $\g$ is a ``surjective" $\Z_2$-contraction, 
its symmetric invariants are known
\cite[Theorem~4.5]{Dima06}.  The generators of  $ \cS(\g)^{\g}$ are $\Delta_j^\bullet$ with $2\le j\le n{+}k$.  
Here $\deg \Delta_j^\bullet=j$ and the generators with the degrees $2,3,\ldots,n{-}k$ in 
$\gt{sl}_n{\oplus}\gt{gl}_k$ are $\Delta_{2k{+}2}^\bullet, \Delta_{2k{+}3}^\bullet, \ldots,\Delta_{n{+}k}^\bullet$.  
As the restriction to 
$\gt{sl}_n{\oplus}\gt{gl}_k{+}x$ shows, none of the generators $\Delta_j^\bullet$ 
with $j\ge 2k{+}2$ lies in $\cS(\gt q)$. This means that ${\bf h}_t$ cannot be equal or even proportional over $\C[V^*]^G$ to 
$\Delta_{2k{+}t}^\bullet$ and hence has a more complicated expression. More precisely, 
a product $\Delta_{2k{+}1}^\bullet\Delta_{2k{+}t{-}1}^\bullet$ necessary appears in ${\bf h}_t$ with a non-zero 
coefficient from $\C[V^*]^G$ for $t\ge 2$. 
Since $\deg_V  \Delta_{2k{+}1}^\bullet=2k$, we have $\deg_V {\bf h}_t\ge 4k$ for every $t\ge 2$. 
The ring $\C[V^*]^G$ is freely generated by $k^2$ polynomials of degree two. Therefore, the total sum of degrees over all generators of $\cS(\gt q)^{\gt q}$ is greater than or equal to 
$$
{\bf b}(\gt{sl}_{n{-}k})+ 4k(n{-}k{-}1)+2k^2={\bf b}(\gt q)+2k(n{-}k{-}2).
$$
This contradicts Theorem~\ref{svob}({\sf i}). 
\end{proof}



\subsection{The case $0<k<m$} Here $\C[V^*]^G$ is a polynomial ring if and only if $m\le n$, \cite{gerry}. 
If $n=m$, then $\gt g_x=0$ for generic $x\in V^*$. 
For $m<n$, our  construction of invariants is rather intricate.


Let $\pi_1,\ldots,\pi_{n-1}$ be the fundamental weights of $\gt{sl}_n$. 
We use the standard convention, $\pi_i=\esi_1{+}\ldots{+}\esi_i$, $\esi_n=-\sum\limits_{i=1}^{n{-}1}\esi_i$. 
Recall that for any $t$, $1\le t <n$, 
$\Lambda^t \C^n$ is irreducible with the highest weight $\pi_t$. 
Let $\{e_1,\ldots,e_n\}$ be a basis of $\C^n$ such that 
each $e_i$ is a weight vector and  
$\ell_t:=e_1\wedge\ldots \wedge e_t$ is a highest weight vector of 
$\Lambda^t \C^n$. 
Clearly 
$\Lambda^t \C^n \subset \cS^t (t \C^n)$. 
Write $n-k=d(m-k)+r$ with $0 < r \le (m-k)$. Let $\varphi\!: m\C^n \to \Lambda^m \C^n$ be 
a non-zero $G$-equivariant map, which is unique up to a scalar.  In case $r \ne m{-}k$, for 
any subset $I\subset\{1,\ldots,m\}$ with $|I|=k+r$, let $\varphi_I\!: m\C^n\to (k{+}r)\C^n\to \Lambda^{k{+}r} \C^n$ be 
the corresponding (almost) canonical map. By the same principle we construct 
$\tilde\varphi\!: k(\C^n)^*\to \Lambda^k (\C^n)^*$. 

Let us consider the tensor product $\mathbb W:=(\Lambda^m \C^n)^{{\otimes}d}{\otimes}\Lambda^{k{+}r}\C^n$
and its weight subspace $\mathbb W_{d\pi_k}$. One can easily see that $\mathbb W_{d\pi_k}$
contains a unique up to a scalar non-zero highest  weight vector, namely 
$$
w_{d\pi_k}=\sum_{\sigma\in S_{n{-}k}} {\rm sgn}(\sigma)  (\ell_k\wedge e_{\sigma(k{+}1)}\wedge\ldots \wedge e_{\sigma(m)})\otimes 
\ldots \otimes (\ell_k\wedge e_{\sigma(n{-}r{+}1)} \ldots \wedge e_{\sigma(n)}).
$$
This means that $\mathbb W$ contains a unique copy of $V_{d\pi_k}$, where $V_{d\pi_k}$
is an irreducible $\gt{sl}_n$-module with the highest weight $d\pi_k$. 
We let $\rho$ denote the representation of $\gt{gl}_n$ on $\Lambda^m\C^n$ and 
$\rho_r$ the representation of $\gt{gl}_n$ on $\Lambda^{k+r}\C^n$. 
Let $\xi=A+v+w$ be a point in $\gt q^*$. (It is assumed that $A\in\gt{sl}_n$.) 
Finally let $(\,\,,\,)$ denote 
a non-zero scalar product between $\mathbb W$ and $\cS^d(\Lambda^k(\C^n)^*)$, which 
is zero on the $\gt{sl}_n$-invariant complement of $V_{d\pi_k}$ in $\mathbb W$. 
Depending on $r$, set
$$
\begin{array}{l}
{\bf F}(\xi):=(\varphi(v){\otimes}\rho(A)^{m{-}k}\varphi(v){\otimes}\rho(A^2)^{m{-}k}\varphi(v){\otimes}\ldots{\otimes} \rho(A^{d})^{m{-}k}\varphi(v),\tilde\varphi(w)^d) \ \text{ for } r=m{-}k; \\
{\bf F}_I(\xi):=(\varphi(v){\otimes}\rho(A)^{m{-}k}\varphi(v){\otimes}\ldots{\otimes} \rho(A^{d{-}1})^{m{-}k}\varphi(v){\otimes}\rho_r(A^d)^{r}\varphi_I(v),\tilde\varphi(w)^d)
\end{array}
$$
for each $I$ as above in case $r<m-k$.  
By the constructions the polynomials ${\bf F}$ and ${\bf F}_I$ are 
$\SL_n$-invariants. 

\begin{lm}\label{lm-3}
The polynomials ${\bf F}$ and ${\bf F}_I$ are $V$-invariants.
\end{lm}
\begin{proof}
We restrict ${\bf F}$ and ${\bf F}_I$ to 
$\gt g^*{+}x$ with $x\in V^*$ generic. 
Changing a basis in $V$ if necessary, we may assume that $x=E_m{+}E_k$. 
If $r<m{-}k$, some of the invariants ${\bf F}_I$ may become linear combinations of such polynomials
under the change of basis, but this does not interfere with  $V$-invariance.     
Now $\varphi(v)$ is a vector of weight $\pi_m$ and
$\tilde\varphi(w)^d$ of weight $-d\pi_k$.  
Notice that $dm{+}(k{+}r)=n{+}kd$.
If $\sum\limits_{i=1}^{n+kd} \lambda_i=d\sum\limits_{i=1}^{k}\esi_i$ and 
each $\lambda_i$ is one of the $\esi_j$, $1\le j\le n$, then in 
$\{\lambda_i\}$ we must have exactly one $\esi_j$ for each $k<j\le n$ and 
$d{+}1$ copies of each $\esi_i$ with $1\le i\le k$. Hence the only summand of $\rho(A^s)^{m{-}k}\varphi(E_m)$ that plays any r\^ole in ${\bf F}$ or ${\bf F}_I$ is 
$\ell_k {\wedge}A^s e_{k+1}{\wedge}\ldots{\wedge}A^s e_m$. Moreover, in 
 $A^s e_{k+1}{\wedge}\ldots{\wedge}A^s e_m$ we are interested only in  
vectors lying  in $\Lambda^{m{-}k}{\rm span}(e_{k{+}1},\ldots,e_n)$.

Let us choose blocks $\alpha, U, \beta$ of $A$ as shown in Figure~\ref{blocks}.
\begin{figure}[htb]
{\setlength{\unitlength}{0.08in}
\begin{center}
\begin{picture}(14,14)(0,0)

\put(0,0){\line(1,0){14}}\put(0,0){\line(0,1){14}}
\put(14,0){\line(0,1){14}}\put(0,14){\line(1,0){14}}

\qbezier[36](0,11)(7,11)(14,11)
\qbezier[36](3,0)(3,7)(3,14)

\qbezier[28](3,7)(8.5,7)(14,7)
\qbezier[28](7,0)(7,5.5)(7,11)

\put(10,3){$\beta$}
\put(4.5,8.5){$\alpha$}
\put(4.5,3){$U$}

\put(-1.3,12){$\left\{ {\parbox{1pt}{\vspace{2.7\unitlength}}}  \right.$}
\put(13.8,8.5){$\left. {\parbox{1pt}{\vspace{3.2\unitlength}}}  \right\}$}
\put(15.3,8.5){{$m{-}k$}}

\put(-2.5,12){{$k$}}
\end{picture}

\end{center} }
\caption{Submatrices of $A\in\gt{sl}_n$}\label{blocks}
\end{figure}
%
%
Then up to a non-zero scalar 
${\bf F}(A,E_m+E_k)$ is the determinant of 
$$
\left( U|\beta U + U\alpha | P_2(\alpha,U,\beta)| \ldots   | P_{d-1}(\alpha,U,\beta)\right), \ \text{ where }
P_s(\alpha,U,\beta)=\sum_{t=0}^{s}  \beta^{t} U \alpha^{s{-}t}.
$$
Each column of $U\alpha$ is a linear combination of the columns of $U$, a similar relation exists between 
$\beta^{t} U\alpha^{s{+}1}$ and $\beta^{t} U\alpha^{s}$.  Therefore 
\begin{equation}\label{bf-F}
{\bf F}(A,E_m+E_k)=\det \left( U|\beta U |\beta^2 U| \ldots | \beta^{d{-}1} U\right) . 
\end{equation}
We have to check that ${\bf F}(\xi{+}A,x)={\bf F}(\xi,x)$ for any $A\in \ad^*(V){\cdot} x$ and any $\xi\in\gt g$,
see Lemma~\ref{V-inv}. Recall that 
$\ad^*(V){\cdot}x=\Ann(\gt g_x)=\gt g_x^{\perp}\subset \gt g$.  In case $x=E_m+E_k$, 
$U$ is zero in each $A\in \gt g_x^{\perp}$
and $\beta$ corresponding to such $A$ is a scalar matrix. Therefore ${\bf F}(\xi+\ad^*(V){\cdot}x,x)={\bf F}(\xi,x)$. 

The case $r<m{-}k$ is more complicated. 
If $\{1,\ldots,k\}\subset I$, then everything works as above and 
$$
{\bf F}_I(A,x)=\det \left( U|\beta U |\beta^2 U| \ldots | \beta^{d{-}2}U| \beta^{d{-}1} U_{\tilde I}\right),
$$
where $I=\tilde I\sqcup\{1,\ldots,k\}$ and $U_{\tilde I}$ is the corresponding submatrix of $U$.  
One has to notice that in the last  polynomial $P_{d-1}(\alpha,U,\beta)$ 
the matrix $U$ is replaced by $U_{\tilde I}$ and 
$\alpha$  by $\alpha_{\tilde I{\times}\tilde I}$.  
We obtain $\binom{m{-}k}{r}$ linearly independent invariants in $\cS(\gt g_x)$. 
Suppose that $\{1,\ldots,k\}\not\subset I$. Then $\rho_I(A^d)^r$ has to move more than 
$r$ vectors $e_i$ with $k{+}1\le i \le m$, which is impossible. Thus,  
${\bf F}_I(A,x)=0$ for such $I$. 
\end{proof}

\begin{thm}\label{divides}
Suppose that $0<k<m <n$ and $m{-}k$ divides $n{-}m$, then 
$\ind\gt g_x=1$ for generic $x\in V^*$ and  $\C[\gt q^*]^Q=\C[V^*]^G[{\bf F}]$ is a polynomial ring, the fundamental 
semi-invariant is equal to  ${\bf F}^{m{-}k{-}1}$.
\end{thm}
\begin{proof}
A generic stabilser $\gt g_x$ is $\gt{sl}_{n{-}m}{\ltimes}(m{-}k)\C^{n{-}m}$. Its ring of symmetric invariants 
is generated by $F=\varphi_x({\bf F})$, see Theorem~\ref{th-0} and Eq.~\eqref{bf-F}. We also have 
$\ind\gt g_x=1$. It remains to see that ${\bf F}$ is  not divisible by 
a non-constant $G$-invariant polynomial on $V^*$. By the construction, ${\bf F}$ is also invariant with respect to the action 
of $\SL_m{\times}\SL_k$. As long as $\rk w=k$, $\rk v=m$ and the upper $k{\times}m$-part of $v$ has rank $k$, 
the  $\SL_n{\times}\SL_m{\times}\SL_k$-orbit of $y=v+w$ contains $x=E_m{+}cE_k$ with $c\in\C^{^\times}$. Here ${\bf F}$ is non-zero on $\gt g{+}y$, up to a non-zero scalar $\varphi_x({\bf F})$ is the same as described by Eq.~\eqref{bf-F}.  
Since $n>m>k$, the complement of $\SL_n{\times}\SL_m{\times}\SL_k{\cdot}(E_m{+}\C^{^\times}\!E_k)$
contains no divisors and ${\bf F}$ is not divisible by any non-constant $G$-invariant 
in $\C[V^*]$. This is enough to conclude that $\C[\gt q^*]^Q=\C[V^*]^G[{\bf F}]$, see 
Theorem~\ref{index-1-sur}.

The singular set $\gt q^*_{\rm sing}$ is $\SL_n{\times}\SL_m{\times}\SL_k$-stable. The above discussion shows that it cannot have a component $\gt g{\times}D$, where $D$ is a divisor in $V^*$. Therefore
${\bf p}$ is a power of ${\bf F}$. In view of Theorem~\ref{svob}({\sf i}), 
${\bf p}={\bf F}^{m{-}k{-}1}$. 
\end{proof}

\begin{thm}\label{n-divides}
Suppose that $0<k<m < n$ and $m{-}k$ does not divide $n{-}m$, then $\C[\gt q^*]^Q$ is not a polynomial ring. 
\end{thm}
\begin{proof}
The reason for this misfortune is  that $\binom{m}{k{+}r}>\binom{m{-}k}{r}$ for $r<m{-}k$. 
One could prove that each ${\bf F}_I$ must be in the set of generators and thereby show that 
$\C[\gt q^*]^Q$  is not a polynomial ring. But we present a different argument.  

Assume that the ring of symmetric invariants is polynomial. It is bi-graded and $\SL_m$ acts on it preserving the 
bi-grading. Since 
$\SL_m$ is reductive, we can assume that there is a set $\{H_1,\ldots,H_s\}$ of bi-homogeneous mixed generators such that $\cS(\gt q)^{\gt q}=\C[V^*]^G[H_1,\ldots,H_s]$ and the $\C$-linear span ${\mathcal H}:={\rm span}(H_1,\ldots,H_s)$  
is $\SL_m$-stable.  The polynomiality implies that 
a generic stabiliser $\gt g_x=\gt{sl}_{n{-}m}{\ltimes}(m{-}k)\C^{n{-}m}$ has a free algebra of symmetric invariants, see Theorem~\ref{svob}({\sf ii}), and by the same statement $\varphi_x$ is surjective. 
This means that $r$ is either $1$ or $m{-}k{-}1$, see Theorem~\ref{th-0},   $s=m{-}k$, and 
$\varphi_x$ is injective on ${\mathcal H}$. 
 Taking our favourite (generic) $x=E_m{+}E_k$, we see that there is $\SL_{m{-}k}$ embedded diagonally into 
$G{\times}\SL_m$, which acts on $\varphi_x({\mathcal H})$ as on 
$\Lambda^r\C^{m{-}k}$.  The group $\SL_{m{-}k}$ acts on ${\mathcal H}$ in the same way. 
Since $m{-}k$ does not divide $n{-}m$, we have $m{-}k\ge 2$. 
The group $\SL_m$ cannot act on 
a irreducible module $\Lambda^r\C^{m{-}k}$ of its non-trivial subgroup $\SL_{m{-}k}$, this is especially obvious in our two cases of 
interest, $r=1$ and $r=m{-}k{-}1$. A contradiction. 
\end{proof}

\begin{conj}
It is very probable that $\C[\gt q^*]^{\gt q}=\C[V^*]^G[\{{\bf F}_I\}]$ for all $n>m>k\ge 1$. 
\end{conj}

\end{document}